\documentclass[a4paper, 12pt, reqno]{amsart}

\pdfoutput=1

\usepackage[utf8]{inputenc}
\usepackage[margin=2.25cm, headsep=0.75cm, footskip=0.75cm]{geometry}
\usepackage{amsmath, amsthm, amssymb, mathtools, physics}
\usepackage[hang, flushmargin]{footmisc}
\usepackage[english]{isodate}
\usepackage{enumitem}
\usepackage[dvipsnames]{xcolor}
\usepackage[colorlinks=true, linkcolor=blue, linkbordercolor=blue, citecolor=red, citebordercolor=red, urlcolor=indigo, linktocpage=true]{hyperref}
\usepackage[capitalise, nameinlink, noabbrev, nosort]{cleveref}
\usepackage{doi}

\definecolor{indigo}{HTML}{492DA5}

\allowdisplaybreaks

\setenumerate{listparindent=\parindent}

\makeatletter\g@addto@macro\bfseries{\boldmath}\makeatother

\makeatletter
\let\origsection\section
\renewcommand\section{\@ifstar{\starsection}{\nostarsection}}
\newcommand\sectionspace{\vspace{0.5ex}}
\newcommand\nostarsection[1]{\sectionspace\origsection{#1}\sectionspace}
\newcommand\starsection[1]{\sectionspace\origsection*{#1}\sectionspace}
\makeatother

\setlist[enumerate]{font=\normalfont}
\crefname{enumi}{}{}
\crefname{enumii}{}{}

\newlist{caselist}{enumerate}{1}
\setlist[caselist, 1]{label=\textbf{Case~\arabic*:}, ref=\arabic*, wide, labelwidth=0em, labelindent=0em, topsep=1ex, itemsep=1ex}
\crefname{caselisti}{Case}{Cases}

\numberwithin{equation}{section}
\crefname{equation}{Equation}{Equations}

\crefname{formula}{Formula}{Formulae}

\crefname{condition}{condition}{conditions}
\creflabelformat{condition}{#2#1#3}

\newtheorem{theorem}{Theorem}[section]

\crefname{thm}{Theorem}{Theorems}

\newtheorem{thmx}{Theorem}

\crefname{thmx}{Theorem}{Theorems}

\newtheorem{lemma}[theorem]{Lemma}
\crefname{lemma}{Lemma}{Lemmas}

\newtheorem{prop}[theorem]{Proposition}
\crefname{prop}{Proposition}{Propositions}

\crefname{cor}{Corollary}{Corollaries}

\theoremstyle{definition}

\newtheorem{definition}[theorem]{Definition}
\crefname{definition}{Definition}{Definitions}

\newtheorem{notation}[theorem]{Notation}
\crefname{notation}{Notation}{Notations}

\theoremstyle{remark}

\newtheorem{remark}[theorem]{Remark}
\crefname{remark}{Remark}{Remarks}

\crefname{example}{Example}{Examples}

\newcommand{\ve}{\varepsilon}
\newcommand{\F}{\mathbb{F}}
\newcommand{\N}{\mathbb{N}}
\newcommand{\T}{\mathbb{T}}
\newcommand{\Z}{\mathbb{Z}}
\newcommand{\EE}{\mathcal{E}}
\newcommand{\GG}{\mathcal{G}}
\newcommand{\II}{\mathcal{I}}
\newcommand{\Eo}{\EE^{(0)}}
\newcommand{\Ec}{\EE^{(2)}}
\newcommand{\Go}{\GG^{(0)}}
\newcommand{\Gc}{\GG^{(2)}}
\newcommand{\IG}{\II^\GG}
\newcommand{\IE}{\II^\EE}
\renewcommand{\pb}[1]{p^{(#1)}}
\newcommand{\Bb}[1]{B^{(#1)}}
\newcommand{\Cb}[1]{C^{(#1)}}
\newcommand{\Eb}[1]{\EE_{#1}}
\newcommand{\Gb}[1]{\GG_{#1}}
\newcommand{\pbinv}[1]{p^{(#1^{-1})}}
\newcommand{\Bbinv}[1]{B^{(#1^{-1})}}
\newcommand{\Cbinv}[1]{C^{(#1^{-1})}}
\newcommand{\pw}{\pb{w}}
\newcommand{\Bw}{\Bb{w}}
\newcommand{\Cw}{\Cb{w}}
\newcommand{\Ew}{\Eb{w}}
\newcommand{\Gw}{\Gb{w}}
\newcommand{\pwinv}{\pbinv{w}}
\newcommand{\Bwinv}{\Bbinv{w}}

\newcommand{\pd}{\pb{d}}
\newcommand{\Bd}{\Bb{d}}
\newcommand{\Cd}{\Cb{d}}
\newcommand{\EEb}{\Eb{b}}
\newcommand{\GGb}{\Gb{b}}
\newcommand{\pdinv}{\pbinv{d}}

\newcommand{\Cdinv}{\Cbinv{d}}
\newcommand{\nbhd}[3]{{#1}_{#3}^{#2}}
\newcommand{\locsec}[3]{{#1}_{#2}^{(#3)}}
\newcommand{\id}{\operatorname{id}}
\newcommand{\Iso}{\operatorname{Iso}}

\newcommand{\restr}[1]{\ensuremath{\vert_{#1}}}
\newcommand{\starx}[1]{\ensuremath{*_{#1}}}
\newcommand{\xstary}[2]{\ensuremath{\,{}_{#1}{*}_{#2}\,}}

\hypersetup{pdfauthor={Becky Armstrong, Abraham C.S. Ng, Aidan Sims, and Yumiao Zhou}, pdftitle={A twist over a minimal \'etale groupoid that is topologically nontrivial over the interior of the isotropy}}

\date{\today}
\title[A topologically nontrivial twist]{A twist over a minimal \'etale groupoid that is topologically nontrivial over the interior of the isotropy}

\author[Armstrong]{Becky Armstrong}
\author[Ng]{Abraham C.S. Ng}
\author[Sims]{Aidan Sims}
\author[Zhou]{Yumiao Zhou}
\address[B.~Armstrong]{School of Mathematics and Statistics, Victoria University of Wellington, Wellington 6012, NEW ZEALAND}
\email{\href{mailto:becky.armstrong@vuw.ac.nz}{becky.armstrong@vuw.ac.nz}}
\address[A.C.S.~Ng]{School of Mathematics and Statistics, The University of Sydney, NSW 2006, AUSTRALIA}
\email{\href{mailto:abraham.ng@sydney.edu.au}{abraham.ng@sydney.edu.au}}
\address[A.~Sims and Y.~Zhou]{School of Mathematics and Applied Statistics, University of Wollongong, NSW 2522, AUSTRALIA}
\email[A.~Sims]{\href{mailto:asims@uow.edu.au}{asims@uow.edu.au}}
\email[Y.~Zhou]{\href{mailto:ymchou1989@outlook.com}{ymchou1989@outlook.com}}

\thanks{This research was funded by a University of Wollongong AEGiS Connect Grant; the Australian Research Council grants DP180100595 and DP200100155; the Deutsche Forschungsgemeinschaft (DFG, German Research Foundation) under Germany's Excellence Strategy -- EXC 2044 -- 390685587, Mathematics M\"unster -- Dynamics -- Geometry -- Structure; the Deutsche Forschungsgemeinschaft (DFG, German Research Foundation) -- Project-ID 427320536 -- SFB 1442; the ERC Advanced Grant 834267 -- AMAREC; and the Marsden Fund of the Royal Society of New Zealand (grant number 21-VUW-156)}

\begin{document}

\begin{abstract}
We present an example of a twist over a minimal Hausdorff \'etale groupoid such that the restriction of the twist to the interior of the isotropy is not topologically trivial; that is, the restricted twist is not induced by a continuous $2$-cocycle.
\end{abstract}

\maketitle

\section{Introduction}

The theory of Cartan subalgebras in operator algebras deals with the question of recognising when an operator algebra together with a distinguished abelian subalgebra arises as the closure of the convolution algebra of compactly supported functions on a groupoid, possibly twisted by cohomological data. The study of this question began with Feldman and Moore \cite{FM1975, FM1977I, FM1977II} in the context of von Neumann algebras. Their theorem says that every Cartan pair of von Neumann algebras arises from a measured equivalence relation $R$ and a measurable circle-valued $2$-cocycle on $R$. The corresponding question for C*-algebras took longer to answer. Renault \cite[Theorem~II.4.15]{Renault1980} obtained an exact analogue of Feldman and Moore's results when the Cartan subalgebra is densely spanned by projections; but more general commutative subalgebras required Kumjian's notion of a \emph{twist} (essentially a principal $\T$-bundle that is also a topological groupoid) over an \'etale groupoid \cite{Kumjian1986}; and Renault \cite{Renault2008} subsequently generalised Kumjian's work, which applied only to second-countable principal groupoids, to the more general situation of second-countable topologically principal groupoids. Raad \cite{Raad2022} later generalised Renault's work to the most general possible situation (see \cite[Theorem~3.1]{ABCCLMR2023}) of twists over effective groupoids.

The key difference between the measurable and topological settings, as recognised by Kumjian, is the existence of circle bundles that are locally trivial but not globally trivial. In all of the constructions described above, there is a natural way to construct from a pair $(A,B)$ a \emph{twist}: a principal circle bundle $\EE$ (measurable in the setting of \cite{FM1975} or topological in the setting of \cite{Kumjian1986, Renault2008}) over the groupoid $\GG$ of germs for the action of the normaliser of $B$ in $A$ on the spectrum of $B$ in such a way that the reduced C*-completion of the convolution algebra of sections of the bundle coincides with $A$. For measurable bundles as in \cite{FM1975} one can always choose a measurable section of the bundle, and then the measurable $2$-cocycle appearing in Feldman and Moore's theorem is the obstruction to this section being a homomorphism. Likewise, when the spectrum of $A$ is totally disconnected, so is the groupoid $\GG$, and so the bundle $\EE$ admits a continuous global section, once again yielding the continuous $2$-cocycle of Renault's result in \cite{Renault1980}. Kumjian's innovation was to describe a C*-algebra constructed directly from the bundle $\EE$ without the need to convert to a $2$-cocycle.

Since then, the theory of twisted groupoid C*-algebras has largely concerned itself with twists and their C*-algebras rather than with continuous $2$-cocycles. In particular, the first author proved an automatic-injectivity theorem \cite[Theorem~6.3]{Armstrong2022} for homomorphisms $\pi$ of C*-algebras of twists $(\EE, \GG)$ over Hausdorff \'etale groupoids, whose key hypothesis is that $\pi$ should be injective on the subalgebra corresponding to the interior of the isotropy in $\GG$. This raises a natural question, and one that the authors of this paper have been asked a number of times over the last few years: ``I'd rather work with cocycles; do I \emph{have} to deal with twists?'' More precisely, is there a twist satisfying the hypotheses of the injectivity theorem of \cite{Armstrong2022} that can't be handled by results like those of \cite{ABS2024} for continuous $2$-cocycles?

The simplest example of a twist that does not arise from a $2$-cocycle is due to Kumjian---communicated privately to the second-named author, and described and generalised in \cite[Section~4]{Kumjian1986}. Muhly and Williams discuss another example of a twist over a principal groupoid that is not isomorphic to the twist arising from a continuous $2$-cocycle in \cite[Example~2.1]{MW1992}; but since our construction is more closely related to Kumjian's, we focus on his example. We give the details in \cref{section: Kumjian's example}, but the idea is as follows. Fix a nontrivial principal $\T$-bundle $p\colon B \to X$. Let $\GG$ be the transformation groupoid for $\Z_2$ interchanging two copies of $X$. The twist $\EE$ comprises a copy of $(X \sqcup X) \times \T$ over $\Go \cong X \sqcup X$, and copies of $B$ and its conjugate bundle $\overline{B}$ over the nontrivial arrows. Multiplication is defined using multiplication in $\T$, the natural actions of $\T$ on $B$ and $\overline{B}$, and the standard isomorphisms $(B \starx{X} \overline{B}) / \T \cong X \times \T$ and $(\overline{B} \starx{X} B) / \T \cong X \times \T$. This twist cannot arise from a $2$-cocycle because it contains a copy of the nontrivial bundle $B$, while twists arising from $2$-cocycles are topologically trivial.

However, this twist is not minimal (each orbit has just two points), and its reduction to the interior of the isotropy of $\GG$ is the trivial twist over $X \sqcup X$, which arises from a (trivial) continuous $2$-cocycle. In particular, this example does not answer the question above. Since twisted groupoid C*-algebras have attracted significant recent interest (\cite{Armstrong2022, ABCCLMR2023, ABS2024, AO2022, BR2021, BCLM2025, BEFPR2021, BFPR2021, DGN2022, DGNRW2020, Li2020, LR2019, Seaton1, Seaton2}), we provide here an example that does answer the question (see \cref{section: our twist example}), establishing the following result.

\begin{thmx} \label{theorem: main}
There exists a twist $\EE$ over a minimal Hausdorff \'etale groupoid $\GG$ such that the induced twist $\IE$ over the interior $\IG$ of the isotropy of $\GG$ does not come from a $2$-cocycle.
\end{thmx}

\section{Preliminaries}

We need some background on principal circle bundles, groupoids, $2$-cocycles, and twists.

\begin{definition} \label{definition: principal circle bundle}
A \emph{(locally trivial) principal circle bundle}---henceforth referred to as a \emph{principal $\T$-bundle}---$(B,X,p)$ consists of locally compact Hausdorff spaces $B$ and $X$, a continuous surjection $p\colon B \to X$, and a continuous left action $\cdot\colon \T \times B \to B$, such that
\begin{enumerate}[label=(\roman*)]
\item the action of $\T$ on $B$ is \emph{fibre-preserving}: $p(z \cdot b) = p(b)$ for all $z \in \T$ and $b \in B$;
\item \label[condition]{condition: principality} $p$ is \emph{principal}: the map
\[
\T \times B \ni (z,b) \mapsto (z \cdot b, b) \in B \xstary{p}{p} B \coloneqq \{ (b_1,b_2) \in B \times B : p(b_1) = p(b_2) \}
\]
is a homeomorphism; and
\item \label[condition]{condition: local triviality} $p$ is \emph{locally trivial}: for each $x \in X$ there exist an open neighbourhood $U^x \subseteq X$ of $x$ and a continuous map $S_x\colon U^x \to B$ such that $p(S_x(u)) = u$ for all $u \in U^x$.
\end{enumerate}
\end{definition}

For each $x \in X$, we define $B_x \coloneqq p^{-1}(x)$. \Cref{condition: principality} implies that for each $b \in B$, the map $\T \ni z \mapsto z \cdot b \in B_{p(b)}$ is a homeomorphism. In particular, $\T$ acts \emph{freely} on $B$, in the sense that $z \cdot b = b \implies z = 1$; and $\T$ acts \emph{fibrewise transitively}, in the sense that $p(b_1) = p(b_2) \implies b_2 \in \T \cdot b_1$.

\begin{notation}
Given two elements $b_1$ and $b_2$ in the same fibre of a principal $\T$-bundle, we write $\langle b_1, b_2 \rangle$ for the unique element of $\T$ such that $\langle b_1, b_2 \rangle \cdot b_2 = b_1$.
\end{notation}

\begin{remark} \label{remark: inner product properties}
Fix $z \in \T$, and let $b_1$ and $b_2$ be in the same fibre of a principal $\T$-bundle. Then $\langle z \cdot b_1, b_2 \rangle \cdot b_2 = z \cdot b_1 = z \langle b_1, b_2 \rangle \cdot b_2$, and so $\langle z \cdot b_1, b_2 \rangle = z \langle b_1, b_2 \rangle$. Similarly, $\langle b_1, z \cdot b_2 \rangle = \overline{z} \langle b_1, b_2 \rangle$, and $\langle z \cdot b_1, b_1 \rangle = z = \langle b_1, \overline{z} \cdot b_1 \rangle$. Furthermore, $\overline{\langle b_1, b_2 \rangle} = \langle b_2, b_1 \rangle = \langle \overline{b_1}, \overline{b_2} \rangle$.
\end{remark}

\begin{remark} \label{remark: inner product continuous and surjective}
Given a principal $\T$-bundle $p\colon B \to X$, the map
\[
B \xstary{p}{p} B \ni (b_1, b_2) \mapsto \big(\langle b_1, b_2 \rangle, b_2\big) \in \T \times B
\]
is the inverse of the homeomorphism $\T \times B \ni (z,b) \mapsto (z \cdot b, b) \in B \xstary{p}{p} B$ of \cref{condition: principality}. It follows that the map $B \xstary{p}{p} B \ni (b_1, b_2) \mapsto \langle b_1, b_2 \rangle \in \T$ is continuous and surjective.
\end{remark}

The maps $S_x\colon U^x \to B$ of \cref{condition: local triviality} are called \emph{continuous local sections} of $p$. Given an open set $U \subseteq X$, any continuous local section $S\colon U \to B$ induces a homeomorphism
\[
U \times \T \ni (u,z) \mapsto z \cdot S(u) \in p^{-1}(U).
\]
Conversely, given an open set $U \subseteq X$ and a homeomorphism $\zeta\colon U \times \T \to p^{-1}(U)$ satisfying $p(\zeta(u,z)) = u$ for all $(u,z) \in U \times \T$, the map $U \ni u \mapsto \zeta(u,1) \in B$ is a continuous local section of $p$. A collection $\{ (U^x,\zeta_x) : x \in X \}$ is called a \emph{local trivialisation} of $p$ if each $U^x$ is an open neighbourhood of $x$ and each $\zeta_x$ is a homeomorphism $U^x \times \T \to p^{-1}(U^x)$ satisfying $p(\zeta_x(u,z)) = u$ for all $(u,z) \in U^x \times \T$.

A \emph{continuous global section} for $p$ is a continuous local section with domain $X$. We say that $p$ is \emph{trivial} if it admits a continuous global section; otherwise, we say that $p$ is \emph{nontrivial}. A principal $\T$-bundle $p\colon B \to X$ is trivial if and only if there is a homeomorphism $\phi\colon X \times \T \to B$ that is \emph{$\T$-equivariant} in the sense that $w \cdot \phi(x,z) = \phi(x,wz)$ for all $x \in X$ and $w, z \in \T$, and is \emph{fibre-preserving} in the sense that $p\big(\phi(x,z)\big) = x$ for all $(x,z) \in X \times \T$.

Given a principal $\T$-bundle $p\colon B \to X$, we define the \emph{conjugate bundle} $\overline{p}\colon \overline{B} \to X$ as follows. Let $\overline{B}$ be the set $\{\overline{b} : b \in B\}$, endowed with the unique topology such that the map $b\mapsto \overline{b}$ is a homeomorphism. Define $\overline{p}\colon \overline{B} \to X$ by $\overline{p}(\overline{b}) \coloneqq p(b)$, and define a $\T$-action on $\overline{B}$ by $z \cdot \overline{b} \coloneqq \overline{\overline{z} \cdot b}$. Since $p\colon B \to X$ is a principal $\T$-bundle, $\overline{p}\colon \overline{B} \to X$ is as well.

A \emph{groupoid} $\GG$ is a small category in which every morphism has an inverse. We call $\Go = \{ \gamma\gamma^{-1} : \gamma \in \GG \}$ the \emph{unit space} of $\GG$. The maps $r\colon \gamma \mapsto \gamma\gamma^{-1}$ and $s\colon \gamma \mapsto \gamma^{-1}\gamma$ are called the \emph{range} and \emph{source} maps $\GG \to \Go$. We define $\Gc \coloneqq \{ (\alpha, \beta) \in \GG \times \GG : s(\alpha) = r(\beta) \}$; its elements are called \emph{composable pairs}. We call $\GG$ a \emph{Hausdorff groupoid} if it carries a locally compact Hausdorff topology such that $\Gc \ni (\alpha,\beta) \mapsto \alpha\beta^{-1} \in \GG$ is continuous. It is \emph{\'etale} if $r,s\colon \GG \to \Go$ are local homeomorphisms. It is \emph{minimal} if, for each $x \in \Go$, the orbit $r(s^{-1}(x))$ is dense in $\Go$. The \emph{isotropy} of $\GG$ is the subgroupoid
\[
\Iso(\GG) = \{ \gamma \in \GG : r(\gamma) = s(\gamma) \},
\]
which contains $\Go$. We write $\IG$ for the interior of $\Iso(\GG)$. If $\GG$ is \'etale, $\Go$ is open \cite[Lemma~8.4.2]{Sims2020}, and so $\Go \subseteq \IG$. We say that $\GG$ is \emph{effective} if $\IG = \Go$.

We use the definition of a twist given in \cite[Definition~3.1]{Boenicke2021} (see also \cite[Remark~2.6]{Armstrong2022}). A \emph{twist} $(\EE, \imath, \pi)$ over a Hausdorff \'etale groupoid $\GG$ is a sequence
\[
\Go \times \T \xrightarrow{\imath} \EE \xrightarrow{\pi} \GG
\]
such that $\EE$ is a Hausdorff groupoid, $\imath$ and $\pi$ are groupoid homomorphisms, $\imath$ is a homeomorphism onto $\pi^{-1}(\Go)$, $\pi$ is continuous, open, and surjective, and the extension is \emph{central}: $\imath(r(\ve),z) \ve = \ve \imath(s(\ve),z)$ for all $\ve \in \EE$ and $z \in \T$. We identify $\Go$ with $\Eo = \imath\big(\Go \times \{1\}\big)$. For $x \in \Go$ and $z \in \T$, we have $\pi(\imath(x,z)) = x$. Centrality implies that the actions $\T \curvearrowright \EE$ and $\EE \curvearrowleft \T$ such that $z \cdot \ve \coloneqq \imath(r(\ve),z) \ve$ and $\ve \cdot z \coloneqq \ve \imath(s(\ve),z)$ coincide. By \cite[Proposition~3.4]{Boenicke2021}, $\pi$ admits continuous local sections, so $\pi\colon \EE \to \GG$ is a principal $\T$-bundle. Twists $(\EE_1,\imath_1,\pi_1)$ and $(\EE_2,\imath_2,\pi_2)$ are \emph{isomorphic} if there is an isomorphism $\varphi\colon \EE_1 \to \EE_2$ such that $\varphi \circ \imath_1 = \imath_2$ and $\pi_2 \circ \varphi = \pi_1$.

A \emph{continuous $2$-cocycle} $\varsigma\colon \Gc \to \T$ is a continuous map satisfying the \emph{$2$-cocycle identity}:
\[
\varsigma(\alpha, \beta) \, \varsigma(\alpha\beta, \gamma) = \varsigma(\alpha, \beta\gamma) \, \varsigma(\beta, \gamma)\qquad\text{ whenever $(\alpha,\beta), (\beta,\gamma) \in \Gc$.}
\]
Every continuous $2$-cocycle $\varsigma\colon \Gc \to \T$ yields a twist $\EE_\varsigma$ over $\GG$ as follows. Let $\EE_\varsigma \coloneqq \GG \times \T$ with the product topology. Define operations by
\[
(\alpha,w)(\beta,z) \coloneqq \big(\alpha\beta, \, \varsigma(\alpha,\beta) zw\big) \qquad \text{and} \qquad (\alpha,w)^{-1} \coloneqq \big(\alpha^{-1}, \, \overline{\varsigma(\alpha, \alpha^{-1})} \overline{w}\big).
\]
Define $\imath_\varsigma\colon \Go \times \T \to \EE_\varsigma$ by $\imath_\varsigma(x,z) \coloneqq (x,z)$, and define $\pi_\varsigma\colon \EE_\varsigma \to \GG$ by $\pi_\varsigma(\gamma,z) \coloneqq \gamma$. Then $(\EE_\varsigma, \imath_\varsigma, \pi_\varsigma)$ is a twist over $\GG$.

A twist $(\EE,\imath,\pi)$ over $\GG$ is \emph{topologically trivial} if $\pi$ admits a continuous global section. Given a continuous global section $S\colon \GG \to \EE$, there is a continuous $2$-cocycle $\varsigma\colon \Gc \to \T$ such that $\imath(r(\alpha),\sigma(\alpha,\beta)) = S(\alpha)S(\beta)S(\alpha\beta)^{-1}$ for all $(\alpha,\beta) \in \Gc$, and the twist $(\EE_\varsigma, \imath_\varsigma, \pi_\varsigma)$ as defined above is isomorphic to $(\EE,\imath,\pi)$; we say that $\EE$ is \emph{induced by a $2$-cocycle}.

\section{Kumjian's example}
\label{section: Kumjian's example}

In this section, we present Kumjian's example of a twist that does not arise from a continuous $2$-cocycle (see \cite[Section~4]{Kumjian1986}). While, as discussed in the introduction, Kumjian's construction does not provide an example of a situation in which the uniqueness theorem of \cite{Armstrong2022} is applicable but a similar theorem for effective groupoids (such as \cite[Theorem~10.2.7]{Sims2020} or \cite[Corollary~4.9]{Renault1991}) would not suffice, it conveys the central idea of our construction later.

Kumjian's example is as follows. Let $p\colon B \to X$ be a nontrivial principal $\T$-bundle over a locally compact Hausdorff space $X$ (for example, the Hopf fibration), and let $\overline{p}\colon \overline{B} \to X$ be the conjugate bundle. Define
\[
B \starx{X} \overline{B} \coloneqq \{ (b, \overline{c}) \in B \times \overline{B} : p(b) = \overline{p}(\overline{c}) \},
\]
and let $(B \starx{X} \overline{B}) / \T$ be the quotient of $B \starx{X} \overline{B}$ by the equivalence relation given by $(z \cdot b, \overline{c}) \sim (b, z \cdot \overline{c})$ for all $b \in B$, $\overline{c} \in \overline{B}$, and $z \in \T$. Define $(\overline{B} \starx{X} B) / \T$ analogously.

Let $R_2$ be the full equivalence relation on the two-point set $\{0,1\}$, regarded as a discrete groupoid with two units; we identify the unit space of $R_2$ with $\{0,1\}$. So $R_2$ has two elements that are not units: the element $(0,1)$ with range $0$ and source $1$, and its inverse $(1,0)$ with range $1$ and source $0$. Regarding $X$ as a topological groupoid consisting entirely of units, let $G \coloneqq X \times R_2$, the product groupoid. So $G$ is a locally compact Hausdorff \'etale groupoid with unit space $G^{(0)} = X \times R_2^{(0)} = X \times \{0,1\}$, and remaining elements $X \times \{(0,1), (1,0)\}$, with range and source maps given by $r(x,a) = (x,r(a))$ and $s(x,a) = (x,s(a))$, and with multiplication given by $(x,a)(x,b) = (x,ab)$ whenever $(a,b) \in R_2^{(2)}$. Now consider the set
\[
E \coloneqq \left(X \times \T \times \{0\}\right) \sqcup \left(B \times \{(0,1)\}\right) \sqcup \left(\overline{B} \times \{(1,0)\}\right) \sqcup \left(X \times \T \times \{1\}\right).
\]
Let $E^{(0)} = X \times \{0,1\}$. We must define structure maps on $E$. For this, throughout what follows, we write $\gamma$ for an element of $G$, we write $x$ for an element of $X$, we write $w,z$ for elements of $\T$, we write $i$ for an element of $\{0,1\}$, and we write $b, c$ for elements of $B$, so $\overline{b}, \overline{c}$ are the corresponding elements of $\overline{B}$. We define $r,s\colon E \to E^{(0)}$ and $\cdot^{-1}\colon E \to E$ by
\begin{align*}
r(x,z,i) &= (x,i), \qquad & r(b,(0,1)) &= (p(b),0), \qquad & r(\overline{b}, (1,0)) &= (\overline{p}(\overline{b}), 1) = (p(b),1), \\
(x,z,i)^{-1} &= (x,\overline{z},i), \qquad & (b,(0,1))^{-1} &= (\overline{b},(1,0)), \qquad & (\overline{b},(1,0))^{-1} &= (b,(0,1)),
\end{align*}
and $s(\gamma) = r(\gamma^{-1})$. Let $\theta_l\colon (B \starx{X} \overline{B}) / \T \to \T$ and $\theta_r\colon (\overline{B} \starx{X} B) / \T \to \T$ be the maps $\theta_l([z \cdot b, \overline{b}]) \coloneqq z$ and $\theta_r([\overline{b}, z \cdot b]) \coloneqq z$, and define multiplication on $E$ by
\begin{gather*}
(x,w,i)(x,z,i) = (x, wz, i),\\
(p(b), w, 0)(b, (0,1)) = (w \cdot b, (0,1)) = (b, (0,1))(p(b), w, 1), \\
(\overline{p}(\overline{b}), w, 1)(\overline{b}, (1,0)) = (w \cdot \overline{b}, (1,0)) = (\overline{b}, (1,0))(\overline{p}(\overline{b}), w, 0),\\
(b, (0,1))(\overline{c}, (1,0)) = (p(b), \theta_l([b, \overline{c}]), 0), \quad \text{ and } \quad (\overline{c}, (1,0))(b, (0,1)) = (p(b), \theta_r([\overline{c}, b]), 1).
\end{gather*}
These operations make $E$ into a groupoid. Define $\imath\colon \T \times \{0,1\} \to E$ to be the inclusion map, and define $\pi\colon E \to G$ by
\[
\pi(x,z,i) = (x,i), \qquad \pi(b, (0,1)) = (p(b),(0,1)), \quad \text{ and } \quad \pi(\overline{b},(1,0)) = (\overline{p}(\overline{b}),(1,0)).
\]
Then we obtain a twist
\[
G^{(0)} \times \T \simeq X \times \T \times \{0,1\} \xrightarrow{\imath} E \xrightarrow{\pi} G.
\]
We claim that this twist does not arise from a continuous $\T$-valued $2$-cocycle on $G$. To see why, suppose otherwise. Then there is a continuous global section $\Sigma\colon G \to E$. Restriction of this section to the subspace $X \times \{(0,1)\}$ gives a section
\[
\Sigma\restr{X \times \{(0,1)\}}\colon X \times \{(0,1)\} \to B \times \{(0,1)\}.
\]
Let $\pi_1\colon B \times \{(1,0)\} \to B$ be the projection map, and define $i_1\colon X \to X \times \{(0,1)\}$ by $i_1(x) \coloneqq (x, (0,1))$. Then $\pi_1 \circ \Sigma \circ i_1$ is a continuous global section of the bundle $p\colon B \to X$. This contradicts the nontriviality of $p\colon B \to X$. So $E$ is not topologically trivial. However,
\[
\Iso(G) = \{ \gamma \in G : r(\gamma) = s(\gamma) \} = X \times \{0,1\} = G^{(0)},
\]
so the restriction of the twist to the interior of the isotropy is trivial.

\section{The proof of \texorpdfstring{\cref{theorem: main}}{Theorem A}}
\label{section: our twist example}

\subsection{Defining the twisted groupoid as a topological space}
\label{subsection: nontrivial T-bundle}

Let $p\colon B \to X$ be any nontrivial principal $\T$-bundle over a locally compact Hausdorff space $X$ (see \cref{definition: principal circle bundle}) such that there exists a minimal homeomorphism $\sigma\colon X \to X$. For example, take $\sigma$ to be an irrational rotation map on $X = \T^2$: by \cite[Theorem~6.22]{Morita2001}, since $\T^2$ has second integral cohomology group $\Z$, it admits a nontrivial principal $\T$-bundle $p\colon B \to \T^2$. For each $x \in X$, let $B_x$ be the fibre $p^{-1}(x)$.

Let $\F_2 = \langle a, a^{-1}, b, b^{-1} \rangle$ be the free group with two generators $a$ and $b$. For each $w \in \F_2$, we will define an associated principal $\T$-bundle $\pw\colon \Bw \to X$, and we will then define a groupoid structure on the disjoint union $\bigsqcup_{w \in \F_2} \Bw$.

Let $\alpha$ be the group action of $\F_2$ on $X$ that is defined on generators by
\[
\alpha_a = \sigma \quad \,\text{and}\, \quad \alpha_b = \id_X.
\]
Let $\ve$ be the identity in $\F_2$---the empty word. Define $\pb{\ve}\colon \Cb{\ve} \to X$ and $\pb{a}\colon \Cb{a} \to X$ each to be the trivial bundle $X \times \T \to X$ with $\T$-action $z \cdot (x,w) \coloneqq (x,zw)$. Let $\pb{a^{-1}}\colon \Cb{a^{-1}} \to X$ be the trivial bundle $X \times \T \to X$ with the conjugate $\T$-action $z \cdot (x,w) = (x, \overline{z} w)$. Let $\pb{b}\colon \Cb{b} \to X$ be the nontrivial bundle $p\colon B \to X$, and let $\pb{b^{-1}}\colon \Cb{b^{-1}} \to X$ be the conjugate bundle $\overline{\pb{b}}\colon \overline{\Cb{b}} \to X$.

\begin{notation} \label{notation: bar}
For $d \in \{b,b^{-1}\}$, if $c \in \Cd$, then $\overline{c}$ denotes the copy of $c$ in $\Cdinv$. And for $d \in \{ \ve, a, a^{-1} \}$, if $(x,z) \in \Cd$, then we define $\overline{(x,z)} \coloneqq (\alpha_d(x),\overline{z}) \in \Cdinv$. So for all $d \in \{ \ve, a, a^{-1}, b, b^{-1} \}$ and $c \in \Cd$, we have $\overline{\overline{c}} = c$, and
\begin{equation} \label{equation: pdinv c bar}
\pdinv(\overline{c}) = \alpha_d\big(\pd(c)\big).
\end{equation}
Moreover, the map $\Cd \ni c \mapsto \overline{c} \in \Cdinv$ is a homeomorphism.
\end{notation}

For $w \in \F_2 {\setminus} \{\ve\}$, let $\abs{w}$ denote the word length of $w$, and write $w = w_1 w_2 \dotsm w_{\abs{w}}$; we define $\abs{\ve} = 0$. For $\abs{w} \le 1$, we define $\Bw \coloneqq \Cw$. For $\abs{w} \ge 2$, let $\Bw$ be the quotient of
\begin{align*}
\Cw \coloneqq \big\{ (c_1, \dotsc, c_{\abs{w}}) :\ &c_i \in \Cb{w_i} \text{ for each } i \in \{1, \dotsc, \abs{w}\}, \,\text{ and} \\
& \pb{w_i}(c_i) = \alpha_{w_{i+1}}\big(\pb{w_{i+1}}(c_{i+1})\big) \text{ for each } i \in \{1, \dotsc, \abs{w} - 1\} \big\}
\end{align*}
by the equivalence relation
\[
(c_1, \dotsc, z \cdot c_i, \dotsc, c_{\abs{w}}) \sim (c_1, \dotsc, z \cdot c_j, \dotsc, c_{\abs{w}}), \quad \text{for all } z \in \T, \ i,j \in \{1,\dotsc, \abs{w}\}.
\]
We denote the equivalence class of $c = (c_1, \dotsc, c_{\abs{w}}) \in \Cw$ by $[c] = [c_1, \dotsc, c_{\abs{w}}] \in \Bw$. We give $\Cw \subseteq \prod_{i=1}^{\abs{w}} \Cb{w_i}$ the subspace topology, and $\Bw$ the quotient topology.

\begin{prop} \label{proposition: Bw principal T-bundle}
For each $w \in \F_2 {\setminus} \{\ve\}$, $\Bw$ is a locally compact Hausdorff space, and the map $\pw\colon \Bw \to X$ given by
\[
\pw\big([c_1, \dotsc, c_{\abs{w}}]\big) = \pb{w_{\abs{w}}}(c_{\abs{w}})
\]
defines a (locally trivial) principal $\T$-bundle over $X$ with $\T$-action given by
\[
z \cdot [c_1, \dotsc, c_{\abs{w}}] \coloneqq [c_1, \dotsc, z \cdot c_{\abs{w}}], \quad \text{for } z \in \T.
\]
Moreover, for each $[c] = [c_1, \dotsc, c_{\abs{w}}] \in \Bw$ and $i \in \{1, \dotsc, \abs{w} - 1\}$, we have
\begin{equation} \label{equation: p^(w_i)(c_i)}
\pb{w_i}(c_i) = \alpha_{w_{i+1} \dotsm w_{\abs{w}}}\big(\pb{w_{\abs{w}}}(c_{\abs{w}})\big) = \alpha_{w_{i+1} \dotsm w_{\abs{w}}}\big(\pw([c])\big).
\end{equation}
\end{prop}

Before we prove \cref{proposition: Bw principal T-bundle}, we need the following technical lemma.

\begin{lemma} \label{lemma: Cw LCH}
For each $w \in \F_2$, $\Cw$ is a locally compact Hausdorff space.
\end{lemma}

\begin{proof}
When $w = \ve$, we have $\Cw = X \times \T$, which is a locally compact Hausdorff space. If $w \ne \ve$, and $\big((c_1^i, \dotsc, c_{\abs{w}}^i)\big)_{i \in I}$ is a net in $\Cw$ that converges to $(c_1, \dotsc, c_{\abs{w}}) \in \prod_{j=1}^{\abs{w}} \Cb{w_j}$, then continuity of $\sigma$ and of each $\pb{w_j}$ ensure that $(c_1, \dots, c_{\abs{w}}) \in \Cw$. So $\Cw$ is a closed subspace of $\prod_{j=1}^{\abs{w}} \Cb{w_j}$ and hence is locally compact and Hausdorff.
\end{proof}

\begin{proof}[Proof of \cref{proposition: Bw principal T-bundle}]
Fix $w \in \F_2 {\setminus} \{\ve\}$. We begin by showing that $\Bw$ is a locally compact Hausdorff space. For each $n \in \N {\setminus} \{0\}$, define
\[
K_n \coloneqq \big\{ (z_1, \dotsc, z_n) \in \T^n : z_1 \dotsm z_n = 1 \big\} = \big\{ (z_1, \dotsc, z_{n-1}, \overline{z_1 \dotsm z_{n-1}}) : z_1, \dotsc, z_{n-1} \in \T \big\}.
\]
Then $K_n \simeq \T^{n-1}$ is compact for $n \in \N {\setminus} \{0\}$. Also, $K_{\abs{w}}$ is a group under coordinatewise multiplication, and acts coordinatewise on $\Cw$. For $c = (c_1, \dotsc, c_{\abs{w}}) \in \Cw$, \[
\Bw \ni [c] = \big\{ (z_1 \cdot c_1, \dotsc, z_{\abs{w}} \cdot c_{\abs{w}}) : (z_1, \dotsc, z_{\abs{w}}) \in K_{\abs{w}} \big\} = K_{\abs{w}} \cdot c.
\]
Thus $\Bw$ is the quotient of $\Cw$ by the action of $K_{\abs{w}}$. Since $\Cw$ is locally compact Hausdorff by \cref{lemma: Cw LCH} and $K_{\abs{w}}$ is compact, \cite[Proposition~2 and Corollary~1 of Section~III.4.1 and Proposition~3 of Section~III.4.2]{Bourbaki1995} imply that $\Bw$ is locally compact Hausdorff.

We now show that $\pw\colon \Bw \to X$ is a principal $\T$-bundle satisfying \cref{equation: p^(w_i)(c_i)}. Since $\pb{w_{\abs{w}}}\colon \Cb{w_{\abs{w}}} \to X$ is a continuous surjection, the map $\Cw \ni (c_1, \dotsc, c_{\abs{w}}) \mapsto \pb{w_{\abs{w}}}(c_{\abs{w}}) \in X$ is a continuous surjection, and it follows that it descends to a continuous surjection on the quotient $\Bw$ of $\Cw$.

Similarly, since the action of $\T$ on $\Cb{w_{\abs{w}}}$ is continuous and fibre-preserving, the formula $z \cdot [c_1, \dotsc, c_{\abs{w}}] \coloneqq [c_1, \dotsc, z \cdot c_{\abs{w}}]$ defines a continuous fibre-preserving action of $\T$ on $\Bw$.

For each $[c] = [c_1, \dotsc, c_{\abs{w}}] \in \Bw$ and $i \in \{1, \dotsc, \abs{w} - 1\}$, we have
\[
\pb{w_i}(c_i) = \alpha_{w_{i+1}}\big(\pb{w_{i+1}}(c_{i+1})\big)
\]
by the definition of $\Bw$, and \cref{equation: p^(w_i)(c_i)} then follows by the definition of $\pw$.

We show that $\pw\colon \Bw \to X$ is principal. Define $\varphi\colon \T \times \Bw \to \Bw \xstary{\pw}{\pw} \Bw$ by $\varphi(z,[c]) \coloneqq (z \cdot [c], [c])$. Then $\varphi$ is continuous because $\T$ acts continuously on $\Bw$. For surjectivity of $\varphi$, fix $[c] = [c_1, \dotsc, c_{\abs{w}}] \in \Bw$ and $[c'] = [c'_1, \dotsc, c'_{\abs{w}}] \in \Bw$ with $\pw([c]) = \pw([c'])$. By \cref{equation: p^(w_i)(c_i)}, each $\pb{w_i}(c_i) = \pb{w_i}(c'_i)$, so there is a unique element $\langle c_i, c_i' \rangle \in \T$ such that $\langle c_i, c_i' \rangle \cdot c_i' = c_i$. We have
\[
\textstyle \big(\prod_{i=1}^{\abs{w}} \langle c_i, c'_i \rangle \big) \cdot [c'] = \big[ \langle c_1, c'_1 \rangle \cdot c'_1, \dotsc, \langle c_{\abs{w}}, c'_{\abs{w}} \rangle \cdot c'_{\abs{w}} \big] = [c_1, \dotsc, c_{\abs{w}}] = [c],
\]
giving $\varphi\big(\prod_{i=1}^{\abs{w}} \langle c_i, c'_i \rangle, \, [c'] \big) = ([c], [c'])$, so $\varphi$ is surjective. To see that $\varphi$ is injective, suppose that $\varphi(z,[c]) = \varphi(z',[c'])$. Then $[c] = [c']$, and $z \cdot [c] = z' \cdot [c'] = z' \cdot [c]$. Thus $(c_1, \dotsc, z \cdot c_{\abs{w}}) \sim (c_1, \dotsc, z' \cdot c_{\abs{w}})$, and so $z \cdot c_{\abs{w}} = z' \cdot c_{\abs{w}}$. Since $\pb{w_{\abs{w}}}\colon \Cb{w_{\abs{w}}} \to X$ is a principal $\T$-bundle, $\T$ acts freely on $\Cb{w_\abs{w}}$. Hence $z = z'$, and so $\varphi$ is injective. Therefore, $\varphi$ is a continuous bijection with inverse given by
\[\textstyle
\varphi^{-1}([c],[c']) = \big(\prod_{i=1}^{\abs{w}} \langle c_i, c'_i \rangle, \, [c'] \big).
\]
Since $(c_i, c'_i) \mapsto \langle c_i, c'_i \rangle$ is continuous for each $i \in \{1, \dotsc, \abs{w}\}$ and multiplication in $\T$ is continuous, $\varphi^{-1}$ is continuous. So $\varphi$ is a homeomorphism, and $\pw\colon \Bw \to X$ is principal.

To see that $\pw\colon \Bw \to X$ is locally trivial, fix $x \in X$. For $i \in \{1, \dotsc, \abs{w}\}$, the principal $\T$-bundle $\pb{w_i}\colon \Cb{w_i} \to X$ is locally trivial, so there is an open neighbourhood $\nbhd{U}{x}{i} \subseteq X$ of $\alpha_{w_{i+1} \dotsm w_{\abs{w}}}(x)$ admitting a continuous local section $\locsec{S}{x}{w_i}\colon \nbhd{U}{x}{i} \to \Cb{w_i}$ for $\pb{w_i}$. Then $\pb{w_i} \circ \locsec{S}{x}{w_i} = \id_{\nbhd{U}{x}{i}}$ for each $i \in \{1, \dotsc, \abs{w}\}$. Define
\[
\textstyle U^x \coloneqq \bigcap_{i=1}^{\abs{w}} \alpha_{w_{i+1} \dotsm w_{\abs{w}}}^{-1}\big(\nbhd{U}{x}{i}\big).
\]
Then $U^x \subseteq X$ is an open neighbourhood of $x$ since $\alpha_{w'}$ is a homeomorphism for each $w' \in \F_2$. For each $u \in U^x$ and each $i \in \{1, \dotsc, \abs{w}\}$, define
\[
c^u_i \coloneqq \locsec{S}{x}{w_i}\big(\alpha_{w_{i+1} \dotsm w_{\abs{w}}}(u)\big),
\]
so that
\begin{equation} \label{equation: p^(w_i)(c^u_i)}
\pb{w_i}(c^u_i) = \alpha_{w_{i+1} \dotsm w_{\abs{w}}}(u) = \alpha_{w_{i+1}}\big(\pb{w_{i+1}}(c^u_{i+1})\big).
\end{equation}
Then $\big[c^u_1, \dotsc, c^u_{\abs{w}}\big] \in \Bw$, and we define $\locsec{S}{x}{w}\colon U^x \to \Bw$ by $\locsec{S}{x}{w}(u) \coloneqq \big[c^u_1, \dotsc, c^u_{\abs{w}}\big]$. All the maps involved in the definition of $\locsec{S}{x}{w}$ are continuous on $U^x$, so $\locsec{S}{x}{w}$ is continuous. Moreover, since $c^u_{\abs{w}} = \locsec{S}{x}{w_{\abs{w}}}(u)$, we have
\[
\pw\big(\locsec{S}{x}{w}(u)\big) = \pw\big(\big[c^u_1, \dotsc, c^u_{\abs{w}}\big]\big) = \pb{w_{\abs{w}}}\big(c^u_{\abs{w}}\big) = u,
\]
which proves that $\locsec{S}{x}{w}\colon U^x \to \Bw$ is a continuous local section of $\pw\colon \Bw \to X$. Therefore, $\pw\colon \Bw \to X$ is a (locally trivial) principal $\T$-bundle.
\end{proof}

Let
\begin{equation} \label{equation: topological space EE}
\textstyle \EE \coloneqq \bigsqcup_{w \in \F_2} \Bw = \big\{ (w, c) : w \in \F_2, \, c \in \Bw \big\},
\end{equation}
under the disjoint union topology. For each $w \in \F_2$, define $\Ew \coloneqq \{w\} \times \Bw$. We will give $\EE$ the structure of a topological groupoid.

\subsection{Defining multiplication on the twisted groupoid}
\label{section: multiplication}

In this section we define a continuous partially defined multiplication on the space $\EE$ defined in \cref{equation: topological space EE}, that is $\F_2$-graded in the sense that it carries $\Ew \,\xstary{s}{r}\, \Eb{w'}$ to $\Eb{ww'}$. The basic idea is that if there is cancellation in the product of $w$ and $w'$ in $\F_2$, say $w = uv$ and $w' = v^{-1}u'$, then we can eliminate the corresponding entries in $\Bw$ and $\Bb{w'}$ of composable pairs in $\Ew \times \Eb{w'}$.

For $d, e \in \{ \ve, a, a^{-1}, b, b^{-1} \}$ we describe a Baer-sum-like product $\Bb{d, e}$ of $\Bd$ and $\Bb{e}$, called their \emph{balanced fibred product}. Specifically, we define $\Bb{d,e}$ as the quotient of
\[
\Cb{d,e} \coloneqq \Bd \xstary{\pd}{\alpha_e \circ \pb{e}} \,\Bb{e} = \{ (c_1,c_2) \in \Bd \times \Bb{e} : \pd(c_1) = \alpha_e(\pb{e})(c_2) \}
\]
by the equivalence relation
\begin{equation} \label{equation: equivalence relation}
(c_1, z \cdot c_2) \sim (z \cdot c_1, c_2), \quad \text{for } z \in \T.
\end{equation}
We endow $\Cb{d,e}$ with the subspace topology inherited from $\Cd \times \Cb{e}$, and $\Bb{d,e}$ with the quotient topology. Note that if $d, e \ne \ve$ and $d \ne e^{-1}$, then $\Cb{d,e} = \Cb{de}$ and $\Bb{d,e} = \Bb{de}$. In \cref{lemma: reducing the products} we show that for all $d, e \in \{ \ve, a, a^{-1}, b, b^{-1} \}$, $\Bb{d,e}$ is homeomorphic to $\Bb{de}$.

We denote the equivalence class of $(c_1, c_2) \in \Cb{d,e}$ by $[c_1,c_2] \in \Bb{d,e}$. There is a continuous action of $\T$ on $\Bb{d,e}$ given by
\[
z \cdot [c_1, c_2] \coloneqq [z \cdot c_1, c_2] = [c_1, z \cdot c_2], \quad \text{for } z \in \T \text{ and } [c_1,c_2] \in \Bb{d,e}.
\]

\begin{lemma} \label{lemma: B d dinv}
Fix $d \in \{ \ve, a, a^{-1}, b, b^{-1} \}$. For $c \in \Bd$, let $\overline{c}$ be as in \cref{notation: bar}. Then
\begin{equation} \label{equation: canonical form d dinv}
\Bb{d,d^{-1}} = \big\{ z \cdot [c,\overline{c}] : z \in \T, \, c \in \Cd \big\}.
\end{equation}
\end{lemma}
\begin{proof}
For $c \in \Cd$ and $z \in \T$, we have $\pdinv(\overline{c}) = \alpha_d\big(\pd(c)\big)$ by \cref{equation: pdinv c bar}. So $\pd(z \cdot c) = \pd(c) = \alpha_{d^{-1}}\big(\pdinv(\overline{c})\big)$. Hence $z \cdot [c,\overline{c}] = [z \cdot c,\overline{c}] \in \Bb{d,d^{-1}}$. For the reverse containment, fix $[c_1,c_2] \in \Bb{d,d^{-1}}$. Then $\pd(c_1) = \alpha_{d^{-1}}\big(\pdinv(c_2)\big)$, so $\pdinv(c_2) = \alpha_d\big(\pd(c_1)\big) = \pdinv(\overline{c_1})$. Since $\pdinv\colon \Cdinv \to X$ is a principal $\T$-bundle, there is therefore a unique $z \in \T$ such that $c_2 = z \cdot \overline{c_1}$. Thus $[c_1,c_2] = [c_1, z \cdot \overline{c_1}] = z \cdot [c_1,\overline{c_1}]$, as required.
\end{proof}

We now show that $\Bb{\ve}$ acts as an identity under the balanced fibred product, in the sense that for each $d \in \{ \ve, a, a^{-1}, b, b^{-1} \}$, we have $\Bb{\ve,d} \simeq \Cd \simeq \Bb{d,\ve}$ and $\Bb{d,d^{-1}} \simeq \Cb{\ve}$.

\begin{lemma} \label{lemma: reducing the products}
\begin{enumerate}[label=(\alph*)]
\item \label{item: BFP identity} For each $d \in \{ \ve, a, a^{-1}, b, b^{-1} \}$, the maps
\begin{align*}
\psi_{\ve,d} \colon \Bb{\ve,d} \ni \left[ \big(\alpha_d(\pd(c)), z\big), c \right] &\mapsto z \cdot c \in \Cd = \Bd \quad \text{ and} \\
\psi_{d,\ve} \colon \Bb{d,\ve} \ni \left[ c, \big(\pd(c), z\big) \right] &\mapsto z \cdot c \in \Cd = \Bd
\end{align*}
are $\T$-equivariant homeomorphisms.
\item \label{item: BFP d dinv} For each $d \in \{ a, a^{-1}, b, b^{-1} \}$, the map
\[
\psi_{d,d^{-1}}\colon \Bb{d,d^{-1}} \ni z \cdot [c, \overline{c}] \mapsto z \cdot \big(\alpha_d(\pd(c)),1\big) = z \cdot \big(\pdinv(\overline{c}),1\big) \in \Cb{\ve} = \Bb{\ve}
\]
is a $\T$-equivariant homeomorphism.
\end{enumerate}
\end{lemma}

\begin{proof}
For~\cref{item: BFP identity}, we prove the statement about $\psi_{\ve,d}$; the proof of the statement about $\psi_{d,\ve}$ is similar. Fix $d \in \{ \ve, a, a^{-1}, b, b^{-1} \}$. For $\big((x,z),c\big) \in \Cb{\ve,d}$, we have $x = \pb{\ve}(x,z) = \alpha_d\big(\pd(c)\big)$, so
\[
\Cb{\ve,d} = \left\{ \big((x,z), c\big) : z \in \T, \, c \in \Cd, \, x = \alpha_d\big(\pd(c)\big) \right\}\!.
\]
We claim that the map
\[
\Psi_{\ve,d}\colon \Cb{\ve,d} \ni \left( \big(\alpha_d(\pd(c)), z\big), c \right) \mapsto z \cdot c \in \Cd
\]
respects the equivalence relation~\labelcref{equation: equivalence relation}. Indeed, for $u \in \T$,
\[
\Psi_{\ve,d}\left( \big(\alpha_d(\pd(c)), z\big), u \cdot c \right) = z \cdot (u \cdot c) = uz \cdot c = \Psi_{\ve,d}\left( u \cdot \big(\alpha_d(\pd(c)), z\big), c \right)\!.
\]
Thus $\Psi_{\ve,d}$ descends to a map
\[
\psi_{\ve,d}\colon \Bb{\ve,d} \ni \left[ \big(\alpha_d(\pd(c)), z\big), c \right] \mapsto z \cdot c \in \Cd.
\]
For all $u \in \T$ and $[(x,z),c] \in \Bb{\ve,d}$, we have
\[
\psi_{\ve,d}\big(u \cdot [(x,z),c]\big) = \psi_{\ve,d}([(x,uz),c]) = uz \cdot c = u \cdot (z \cdot c) = u \cdot \psi_{\ve,d}\big([(x,z),c]\big),
\]
so $\psi_{\ve,d}$ is $\T$-equivariant. To see that $\psi_{\ve,d}$ is a continuous surjection, note that $\Psi_{\ve,d}\colon \Cb{d,\ve} \to \Cd$ is continuous because $\T$ acts continuously on $\Cd$. Also, $\Psi_{\ve,d}$ is surjective because $\psi_{\ve,d}\left( (\alpha_d(\pd(c)), 1), c \right) = c$ for all $c \in \Cd$. Let $q_{\ve,d}\colon \Cb{\ve,d} \to \Bb{\ve,d}$ be the quotient map. Since $\psi_{\ve,d} \circ q_{\ve,d} = \Psi_{\ve,d}$, the map $\psi_{\ve,d}$ is a continuous surjection. For injectivity, suppose that
\[
\psi_{\ve,d}\big([(x,u),c]\big) = \psi_{\ve,d}\big([(y,z),e]\big)
\]
for some $[(x,u),c], [(y,z),e] \in \Bb{\ve,d}$. Then $u \cdot c = z \cdot e$, and so
\[
x = \alpha_d\big(\pd(u \cdot c)\big) = \alpha_d\big(\pd(z \cdot e)\big) = y.
\]
Hence
\[
[(x,u),c] = u \cdot [(x,1),c] = [(x,1), u \cdot c] = [(y,1), z \cdot e] = z \cdot [(y,1),e] = [(y,z),e],
\]
which proves that $\psi_{\ve,d}$ is injective. Thus $\psi_{\ve,d}$ is bijective with inverse given by
\[
\psi_{\ve,d}^{-1}(c) = \left[ \big(\alpha_d(\pd(c)), 1\big), c \right]\!.
\]
Since $q_{\ve,d}$, $\alpha_d$, and $\pd$ are continuous, so is $\psi_{\ve,d}^{-1}$, so $\psi_{\ve,d}$ is a $\T$-equivariant homeomorphism.

For part~\cref{item: BFP d dinv}, fix $d \in \{ a, a^{-1}, b, b^{-1} \}$, and consider the map
\[
\Psi_{d,d^{-1}}\colon \Cb{d,d^{-1}} \ni (c_1,\overline{c_2}) \mapsto \big( \pdinv(\overline{c_2}), \langle c_1, c_2 \rangle \big) \in \Cb{\ve}.
\]
For all $z \in \T$ and $(c_1,\overline{c_2}) \in \Cb{d,d^{-1}}$, we have $(z \cdot c_1,\overline{c_2}) \sim (c_1, z \cdot \overline{c_2}) = (c_1, \overline{\overline{z} \cdot c_2})$, and by \cref{remark: inner product properties}, $\langle z \cdot c_1, c_2 \rangle = z \langle c_1, c_2 \rangle = \langle c_1, \overline{z} \cdot c_2 \rangle$. Hence
\[
\Psi_{d,d^{-1}}(z \cdot c_1,\overline{c_2}) = \big( \pdinv(\overline{c_2}), \langle z \cdot c_1, c_2 \rangle \big) = \big( \pdinv(\overline{\overline{z} \cdot c_2}), \langle c_1, \overline{z} \cdot c_2 \rangle \big) = \Psi_{d,d^{-1}}(c_1,\overline{\overline{z} \cdot c_2}),
\]
and so $\Psi_{d,d^{-1}}$ is constant on equivalence classes. Thus $\Psi_{d,d^{-1}}$ descends to a map
\[
\psi_{d,d^{-1}}\colon \Bb{d,d^{-1}} \ni [c_1,\overline{c_2}] \mapsto \big( \pdinv(\overline{c_2}), \langle c_1, c_2 \rangle \big) \in \Cb{\ve}.
\]
Recall from \cref{lemma: B d dinv} that $\Bb{d,d^{-1}} = \big\{ z \cdot [c,\overline{c}] : z \in \T, \, c \in \Cd \big\}$, and recall from \cref{equation: pdinv c bar} that $\alpha_d\big(\pd(c)\big) = \pdinv(\overline{c})$ for all $c \in \Cd$. For all $z \in \T$ and $c \in \Cd$, we have $\langle z \cdot c, c \rangle = z$ by \cref{remark: inner product properties}, and it follows that
\[
\psi_{d,d^{-1}}\big(z \cdot [c,\overline{c}]\big) = \psi_{d,d^{-1}}\big([z \cdot c,\overline{c}]\big) = \big( \pdinv(\overline{c}), \langle z \cdot c, c \rangle \big) = z \cdot \big( \pdinv(\overline{c}), 1 \big) = z \cdot \psi_{d,d^{-1}}([c,\overline{c}]),
\]
so $\psi_{d,d^{-1}}$ is $\T$-equivariant. To see that $\psi_{d,d^{-1}}$ is a continuous surjection, first note that $\pdinv\colon \Cdinv \to X$ is a continuous surjection by definition. Hence \cref{remark: inner product continuous and surjective} implies that $\Psi_{d,d^{-1}}\colon \Cb{d,d^{-1}} \to \Cb{\ve}$ is a continuous surjection. Writing $q_{d,d^{-1}}\colon \Cb{d,d^{-1}} \to \Bb{d,d^{-1}}$ for the quotient map and using that $\psi_{d,d^{-1}} \circ q_{d,d^{-1}} = \Psi_{d,d^{-1}}$, it follows that $\psi_{d,d^{-1}}$ is a continuous surjection. To see that $\psi_{d,d^{-1}}$ is injective, suppose that
\[
\psi_{d,d^{-1}}\big(u \cdot [c,\overline{c}]\big) = \psi_{d,d^{-1}}\big(z \cdot [e,\overline{e}]\big),
\]
for some $u, z \in \T$ and $c, e \in \Cd$. Then $\big( \pdinv(\overline{c}), u \big) = \big( \pdinv(\overline{e}), z \big)$, so $u = z$, and since $\pdinv\colon \Cdinv \to X$ is a principal $\T$-bundle, there is a unique element $v \in \T$ such that $\overline{c} = v \cdot \overline{e}$. Hence $c = \overline{v \cdot \overline{e}} = \overline{v} \cdot e$, and so
\[
u \cdot [c,\overline{c}] = z \cdot [\overline{v} \cdot e, v \cdot \overline{e}] = z \overline{v} v \cdot [e,\overline{e}] = z \cdot [e,\overline{e}],
\]
which proves that $\psi_{d,d^{-1}}$ is injective. Thus $\psi_{d,d^{-1}}$ is bijective with inverse given by
\[
\psi_{d,d^{-1}}^{-1}\big(\pdinv(\overline{c}),z\big) = z \cdot [c,\overline{c}].
\]
For any $(x,z) \in \Cb{\ve}$, since $\pdinv$ is locally trivial it has a continuous local section defined on an open neighbourhood of $x$. So $\psi_{d,d^{-1}}^{-1}$ is continuous at $(x,z)$. Thus $\psi_{d,d^{-1}}$ is a $\T$-equivariant homeomorphism.
\end{proof}

\begin{notation} \label{notation: concatenation product}
For $d,e \in \{a, a^{-1}, b, b^{-1}\}$ with $d \ne e^{-1}$, we define $\psi_{d,e} \coloneqq \id_{\Bb{d,e}}$.
\end{notation}

We use the next lemma to see that the maps $\psi_{d,d'}$ determine an associative operation.

\begin{lemma} \label{lemma: associativity}
Fix $d, d' \in \{ \ve, a, a^{-1}, b, b^{-1} \}$. For all $e, e', e^\dagger \in \Cb{\ve}$, $c, c'' \in \Cd$, $c' \in \Cb{d'}$, and $c^\dagger \in \Cdinv$ such that the following operations are defined, we have
\begin{enumerate}[label=(\alph*)]
\item \label{item: associativity ve d ve} $\psi_{d,\ve}\big(\big[\psi_{\ve,d}([e,c]),e'\big]\big) = \psi_{\ve,d}\big(\big[e,\psi_{d,\ve}([c,e'])\big]\big)$;
\item \label{item: associativity d ve d'} $\psi_{d,d'}\big(\big[\psi_{d,\ve}([c,e']),c'\big]\big) = \psi_{d,d'}\big(\big[c,\psi_{\ve,d'}([e',c'])\big]\big)$;
\item \label{item: associativity ve d dinv} $\psi_{d,d^{-1}}\big(\big[\psi_{\ve,d}([e,c]),c^\dagger\big]\big) = \psi_{\ve,\ve}\big(\big[e,\psi_{d,d^{-1}}([c,c^\dagger])\big]\big)$;
\item \label{item: associativity d dinv ve} $\psi_{\ve,\ve}\big(\big[\psi_{d,d^{-1}}([c,c^\dagger]),e^\dagger\big]\big) = \psi_{d,d^{-1}}\big(\big[c,\psi_{d^{-1},\ve}([c^\dagger,e^\dagger])\big]\big)$; and
\item \label{item: associativity d dinv d} $\psi_{\ve,d}\big(\big[\psi_{d,d^{-1}}([c,c^\dagger]),c''\big]\big) = \psi_{d,\ve}\big(\big[c,\psi_{d^{-1},d}([c^\dagger,c''])\big]\big)$.
\end{enumerate}
\end{lemma}

\begin{proof}
Fix $e, e', e^\dagger \in \Cb{\ve}$, $c, c'' \in \Cd$, $c' \in \Cb{d'}$, and $c^\dagger \in \Cdinv$ such that $(e,c) \in \Cb{\ve,d}$, $(c,e') \in \Cb{d,\ve}$, $(e',c') \in \Cb{\ve,d'}$, $(c,c^\dagger) \in \Cb{d,d^{-1}}$, $(c^\dagger,e^\dagger) \in \Cb{d^{-1},\ve}$, and $(c^\dagger,c'') \in \Cb{d^{-1},d}$. Then there exist $x, x', x^\dagger \in X$ and $z, z', z^\dagger \in \T$ such that $e = (x,z)$, $e' = (x',z')$, and $e^\dagger = (x^\dagger, z^\dagger)$. It then follows from \cref{lemma: reducing the products}\cref{item: BFP identity} that
\[
\psi_{\ve,d}([e,c]) = z \cdot c, \quad \psi_{d,\ve}([c,e']) = z' \cdot c, \quad \psi_{\ve,d'}([e',c']) = z' \cdot c', \quad \text{and } \quad \psi_{d^{-1},\ve}([c^\dagger,e^\dagger]) = z^\dagger \cdot c^\dagger.
\]

For part~\cref{item: associativity ve d ve}, we use \cref{lemma: reducing the products}\cref{item: BFP identity} to see that
\[
\psi_{d,\ve}\big(\big[\psi_{\ve,d}([e,c]),e'\big]\big) = \psi_{d,\ve}\big([z \cdot c, e']\big) = zz' \cdot c = \psi_{\ve,d}\big([e, z' \cdot c]\big) = \psi_{\ve,d}\big(\big[e,\psi_{d,\ve}([c,e'])\big]\big).
\]

For part~\cref{item: associativity d ve d'}, we use that $(z' \cdot c, c') \sim (c, z' \cdot c')$ to see that
\[
\psi_{d,d'}\big(\big[\psi_{d,\ve}([c,e']),c'\big]\big) = \psi_{d,d'}\big(\big[z' \cdot c, c'\big]\big) = \psi_{d,d'}\big(\big[c, z' \cdot c'\big]\big) = \psi_{d,d'}\big(\big[c,\psi_{\ve,d'}([e',c'])\big]\big).
\]

For part~\cref{item: associativity ve d dinv}, we use \cref{lemma: reducing the products}\cref{item: BFP identity} to see that
\[
\psi_{d,d^{-1}}\big(\big[\psi_{\ve,d}([e,c]),c^\dagger\big]\big) = \psi_{d,d^{-1}}\big([z \cdot c, c^\dagger]\big) = z \cdot \psi_{d,d^{-1}}\big([c, c^\dagger]\big) = \psi_{\ve,\ve}\big(\big[e,\psi_{d,d^{-1}}([c,c^\dagger])\big]\big).
\]

For part~\cref{item: associativity d dinv ve}, we have
\[
\psi_{\ve,\ve}\big(\big[\psi_{d,d^{-1}}([c,c^\dagger]),e^\dagger\big]\big) = z^\dagger \cdot \psi_{d,d^{-1}}([c,c^\dagger]) = \psi_{d,d^{-1}}\big(\big[c,z^\dagger \cdot c^\dagger\big]\big) = \psi_{d,d^{-1}}\big(\big[c,\psi_{d^{-1},\ve}([c^\dagger,e^\dagger])\big]\big).
\]

For part~\cref{item: associativity d dinv d}, note that since $(c,e') \in \Cb{d,\ve}$, we have $\pd(c) = \pb{\ve}(e') = x'$, and since $(e,c) \in \Cb{\ve,d}$, $(c,c^\dagger) \in \Cb{d,d^{-1}}$, and $(c^\dagger,c'') \in \Cb{d^{-1},d}$, we have
\[
x = \pb{\ve}(e) = \alpha_d\big(\pd(c)\big) = \alpha_d\big(\alpha_{d^{-1}}\big(\pdinv(c^\dagger)\big)\big) = \pdinv(c^\dagger) = \alpha_d\big(\pd(c'')\big).
\]
So $\pd(c'') = \pd(c) = x'$, and by \cref{equation: pdinv c bar}, $\pdinv(\overline{c}) = \alpha_d\big(\pd(c)\big) = \pdinv(c^\dagger) = x$. Since $\pd\colon \Cd \to X$ and $\pdinv\colon \Cdinv \to X$ are principal $\T$-bundles, there are unique elements $u, u^\dagger \in \T$ such that $c'' = u \cdot c$ and $c^\dagger = u^\dagger \cdot \overline{c}$. So \cref{lemma: reducing the products}\cref{item: BFP d dinv} yields
\[
\psi_{d,d^{-1}}([c,c^\dagger]) = (x,u^\dagger) \quad \text{ and } \quad \psi_{d^{-1},d}([c^\dagger,c'']) = (x',u^\dagger u).
\]
It then follows by \cref{lemma: reducing the products}\cref{item: BFP identity} that
\begin{align*}
\psi_{\ve,d}\big(\big[\psi_{d,d^{-1}}([c,c^\dagger]),c''\big]\big) &= \psi_{\ve,d}\big([(x,u^\dagger),u \cdot c]\big) = u^\dagger u \cdot c \\
&= \psi_{d,\ve}\big(\big[c,(x',u^\dagger u)\big]\big) = \psi_{d,\ve}\big(\big[c,\psi_{d^{-1},d}([c^\dagger,c''])\big]\big). \qedhere
\end{align*}
\end{proof}

We now extend the definition of the balanced fibred product to obtain an operation
\[
\big( \Bw, \Bb{w'} \big) \mapsto \Bb{w,w'}
\]
for $w, w' \in \F_2$. For each $w, w' \in \F_2$, let $\Bb{w,w'}$ be the quotient of the set
\[
\Cb{w,w'} \coloneqq \Bw \xstary{\pb{w}}{\alpha_{w'} \circ \pb{w'}} \,\Bb{w'} = \big\{ (c_1,c_2) \in \Bw \times \Bb{w'} : \pw(c_1) = \alpha_{w'}(\pb{w'})(c_2) \big\}
\]
by the equivalence relation
\[
(c_1, z \cdot c_2) \sim (z \cdot c_1, c_2), \quad \text{for } z \in \T.
\]
We endow $\Cb{w,w'}$ with the subspace topology inherited from $\Bw \times \Bb{w'}$, and $\Bb{w,w'}$ with the quotient topology. We denote the equivalence class of $(c_1,c_2) \in \Cb{w,w'}$ by $[c_1,c_2] \in \Bb{w,w'}$. There is a continuous action of $\T$ on $\Bb{w,w'}$ given by
\[
z \cdot [c_1, c_2] \coloneqq [z \cdot c_1, c_2] = [c_1, z \cdot c_2], \quad \text{for } z \in \T \text{ and } [c_1,c_2] \in \Bb{w,w'}.
\]

It follows from \cref{lemma: reducing the products,lemma: associativity} and a long but standard induction argument on $\abs{w} + \abs{w'}$ that for all $w, w' \in \F_2$, there is a unique $\T$-equivariant homeomorphism
\[
\psi_{w,w'}\colon \Bb{w,w'} \to \Bb{ww'}
\]
that coincides with the $\T$-equivariant homeomorphisms defined in \cref{lemma: reducing the products,notation: concatenation product} when $\abs{w}, \abs{w'} \le 1$, and satisfies the following two properties for all $w, w', w'' \in \F_2$:
\begin{enumerate}[label=(\roman*)]
\item for all $[c,c'] \in \Bb{w,w'}$,
\begin{equation} \label{equation: right-fibre-preserving multiplication}
\pb{ww'}\big(\psi_{w,w'}([c,c'])\big) = \pb{w'}(c'); \quad \text{ and}
\end{equation}
\item for all $c \in \Bw$, $c' \in \Bb{w'}$, and $c'' \in \Bb{w''}$ with $[c,c'] \in \Bb{w,w'}$ and $[c',c''] \in \Bb{w',w''}$,
\begin{equation} \label{equation: extended associativity}
\psi_{ww',w''}\big(\big[\psi_{w,w'}([c,c']),c''\big]\big) = \psi_{w,w'w''}\big(\big[c,\psi_{w',w''}([c',c''])\big]\big).
\end{equation}
\end{enumerate}

With $\EE$ as in \cref{equation: topological space EE}, we define the set of composable pairs to be the space
\[
\Ec \coloneqq \big\{ \big((w,c), (w',c')\big) \in \EE \times \EE : (c,c') \in \Cb{w,w'} \big\} \subseteq \EE \times \EE,
\]
under the subspace topology. We define multiplication on $\EE$ by
\[
\Ec \ni \big((w,c), (w',c')\big) \mapsto \big(ww', \psi_{w,w'}([c,c'])\big) \in \EE.
\]
\Cref{equation: extended associativity} implies that this multiplication is associative. It is continuous because each $\psi_{w,w'}$ is a homeomorphism onto the clopen set $\Bb{ww'}$.

\subsection{Defining inversion on the twisted groupoid}

In this section we define a continuous inversion map on $\EE$. Recall from \cref{notation: bar} the definition of the involutive homeomorphism $\Bw \ni c \mapsto \overline{c} \in \Bwinv$ for $w \in \F_2$ with $\abs{w} = 1$. We first extend this map so that it is defined for all $w \in \F_2$ and $c \in \Bw$.

Fix $w = w_1 \dotsb w_{\abs{w}} \in \F_2$, and fix $c = [c_1, \dotsc, c_{\abs{w}}] \in \Bw$ such that $c_i \in \Bb{w_i}$ for each $i \in \{1, \dotsc, \abs{w}\}$. By the definition of $\Bw$ and by \cref{equation: pdinv c bar}, we have
\[
\pb{w_{i+1}^{-1}}(\overline{c_{i+1}}) = \alpha_{w_{i+1}}\big(\pb{w_{i+1}}(c_{i+1})\big) = \pb{w_i}(c_i) = \alpha_{w_i^{-1}}\big(\pb{w_i^{-1}}(\overline{c_i})\big),
\]
and so $(\overline{c_{i+1}}, \overline{c_i}) \in \Bb{w_{i+1}^{-1}, w_i^{-1}}$, for each $i \in \{1, \dotsc, \abs{w}-1\}$. Thus
\[
\overline{c} \coloneqq \big[\overline{c_{\abs{w}}}, \dotsc, \overline{c_1}\big] \in \Bwinv.
\]
For all $z \in \T$,
\[
\overline{z \cdot c} = \big[\overline{z \cdot c_{\abs{w}}}, \dotsc, \overline{c_1}\big] = \big[\overline{z} \cdot \overline{c_{\abs{w}}}, \dotsc, \overline{c_1}\big] = \overline{z} \cdot \overline{c}.
\]
Moreover, $\overline{\overline{c}} = c$, so $\Bw \ni c \mapsto \overline{c} \in \Bwinv$ is a $\T$-contravariant involutive homeomorphism.

\begin{lemma} \label{lemma: c and c bar composable}
For all $w \in \F_2$ and $c \in \Bw$, we have $(c,\overline{c}) \in \Cb{w,w^{-1}}$ and $(\overline{c},c) \in \Cb{w^{-1},w}$, and $\psi_{w,w^{-1}}([c,\overline{c}]) = \big(\pwinv(\overline{c}),1\big)$ and $\psi_{w^{-1},w}([\overline{c},c]) = \big(\pw(c),1\big)$.
\end{lemma}

\begin{proof}
Fix $w = w_1 \dotsb w_{\abs{w}} \in \F_2$, and fix $c = [c_1, \dotsc, c_{\abs{w}}] \in \Bw$ such that $c_i \in \Bb{w_i}$ for each $i \in \{1, \dotsc, \abs{w}\}$. By definition of the principal $\T$-bundle $\pwinv\colon \Bwinv \to X$ of \cref{proposition: Bw principal T-bundle}, we have $\pwinv(\overline{c}) = \pb{w_1^{-1}}(\overline{c_1})$, because $\overline{c} = \big[\overline{c_{\abs{w}}}, \dotsc, \overline{c_1}\big]$ and $\overline{c_1} \in \Bb{w_1^{-1}}$. Using \cref{equation: p^(w_i)(c_i)} for the second equality and \cref{equation: pdinv c bar} for the third equality, we see that
\begin{equation} \label{equation: alpha_w pw c is pwinv c bar}
\alpha_w\big(\pw(c)\big) = \alpha_{w_1}\big(\alpha_{w_2 \dotsb w_{\abs{w}}}\big(\pw(c)\big)\big) = \alpha_{w_1}\big(\pb{w_1}(c_1)\big) = \pb{w_1^{-1}}(\overline{c_1}) = \pwinv(\overline{c}),
\end{equation}
and so $\pw(c) = \alpha_{w^{-1}}\big(\pwinv(\overline{c})\big)$. Thus $(c,\overline{c}) \in \Cb{w,w^{-1}}$ and $(\overline{c},c) \in \Cb{w^{-1},w}$.

We show that $\psi_{w^{-1},w}([\overline{c},c]) = \big(\pw(c),1\big)$ by induction on $\abs{w}$. If $\abs{w} \le 1$, then the claim holds by \cref{lemma: reducing the products}. Suppose that the claim holds for $w \in \F_2$ with $\abs{w} \le n$. Fix $w \in \F_2$ with $\abs{w} = n+1$ and $c \in \Bw$. Let $w' \coloneqq w_1 \dotsb w_n$. Since $\psi_{w',w_{n+1}}\colon \Bb{w',w_{n+1}} \to \Bw$ is a bijection, there exist $c' = [c_1, \dotsc, c_n] \in \Bb{w'}$ and $c_{n+1} \in \Bb{w_{n+1}}$ such that $c = \psi_{w',w_{n+1}}([c',c_{n+1}])$. By the inductive hypothesis, since $\abs{w'} = n$ and $(c_n,c_{n+1}) \in \Cb{w_n,w_{n+1}}$,
\[
\psi_{(w')^{-1},w'}([\overline{c'},c']) = \big(\pb{w'}(c'),1\big) = \big(\pb{w_n}(c_n),1\big) = \big(\alpha_{w_{n+1}}(\pb{w_{n+1}}(c_{n+1})),1\big) \in \Bb{\ve}.
\]
Thus, \cref{lemma: reducing the products}\cref{item: BFP identity} and \cref{equation: extended associativity} yield
\begin{align*}
\psi_{(w')^{-1},w}([\overline{c'},c]) &= \psi_{(w')^{-1},w'w_{n+1}}\big(\big[\overline{c'},c\big]\big) = \psi_{(w')^{-1},w'w_{n+1}}\big(\big[\overline{c'},\psi_{w',w_{n+1}}([c',c_{n+1}])\big]\big) \\
&= \psi_{(w')^{-1}w',w_{n+1}}\big(\big[\psi_{(w')^{-1},w'}([\overline{c'},c']),c_{n+1}\big]\big) \\
&= \psi_{\ve,w_{n+1}}\big(\big[\big(\alpha_{w_{n+1}}(\pb{w_{n+1}}(c_{n+1})),1\big), c_{n+1}\big]\big) = c_{n+1} \in \Bb{w_{n+1}},
\end{align*}
and hence
\begin{align*}
\psi_{w^{-1},w}([\overline{c},c]) &= \psi_{w_{n+1}^{-1},(w')^{-1}w}\big(\big[\overline{c_{n+1}},\psi_{(w')^{-1},w}([\overline{c'},c])\big]\big) \\
&= \psi_{w_{n+1}^{-1},w_{n+1}}([\overline{c_{n+1}},c_{n+1}]) = \big(\pb{w_{n+1}}(c_{n+1}),1\big) = \big(\pw(c),1\big),
\end{align*}
completing the induction. An analogous argument gives $\psi_{w,w^{-1}}([c,\overline{c}]) = \big(\pwinv(\overline{c}),1\big)$.
\end{proof}

For each $(w,c) \in \EE = \bigsqcup_{w \in \F_2} \Bw$, we define
\[
(w,c)^{-1} \coloneq (w^{-1},\overline{c}).
\]

\begin{prop} \label{proposition: EE groupoid}
The set $\EE = \bigsqcup_{w \in \F_2} \Bw$ is a locally compact Hausdorff groupoid under the multiplication
\[
\Ec \ni \big((w,c), (w',c')\big) \mapsto \big(ww', \psi_{w,w'}([c,c'])\big) \in \EE
\]
and the inversion $\EE \ni (w,c) \mapsto (w^{-1},\overline{c}) \in \EE$.
\end{prop}

\begin{proof}
We saw in \cref{section: multiplication} that the multiplication is continuous and associative. For each fixed $w \in \F_2$, the map $\{w\} \times \Bw \ni (w,c) \mapsto (w,c)^{-1} \in \{w^{-1}\} \times \Bwinv$ is a homeomorphism, so inversion is continuous on $\EE = \bigsqcup_{w \in \F_2} \{w\} \times \Bw$.

Fix $(w,c) \in \EE$. Then $\big((w,c)^{-1}\big)^{\!-1} = (w,c)$. Moreover, by \cref{lemma: c and c bar composable}, we have
\[
\big((w,c), (w^{-1},\overline{c})\big), \big((w^{-1},\overline{c}), (w,c)\big) \in \Ec.
\]
It follows from \cref{lemma: c and c bar composable} that
\begin{align*}
r(w,c) \coloneqq (w,c)(w,c)^{-1} &= (w,c)(w^{-1},\overline{c}) = \big(ww^{-1}, \, \psi_{w,w^{-1}}([c,\overline{c}])\big) \\
&= \big(\ve, (\pwinv(\overline{c}),1)\big) = \big(\ve, (\alpha_w(\pw(c)),1)\big),
\end{align*}
and
\[
s(w,c) \coloneqq (w,c)^{-1}(w,c) = (w^{-1},\overline{c})(w,c) = \big(w^{-1}w, \, \psi_{w^{-1},w}([\overline{c},c])\big) = \big(\ve, (\pw(c),1)\big).
\]
Since $\pw\colon \Bw \to X$ is surjective, it follows that
\[
\Eo \coloneqq r(\EE) = s(\EE) = \{\ve\} \times \big(X \times \{1\}\big).
\]
Identify $\Eo$ with $X$. To see that $\Eo$ consists of multiplicative units, fix $x, y \in X$ such that
\[
r(w,c) = (\ve,(x,1)) \quad \text{ and } \quad s(w,c) = (\ve,(y,1)).
\]
Then by \cref{lemma: reducing the products}\cref{item: BFP identity}, we have
\[
r(w,c) \, (w,c) = (\ve,(x,1)) (w,c) = \big(\ve w, \, \psi_{\ve,w}\big([(x,1),c]\big)\big) = (w,c),
\]
and
\[
(w,c) \, s(w,c) = (w,c) (\ve,(y,1)) = \big(w \ve, \, \psi_{w,\ve}\big([c,(y,1)]\big)\big) = (w,c),
\]
as required. Hence $\EE$ is a topological groupoid under the given operations. We know by \cref{proposition: Bw principal T-bundle} that for each $w \in \F_2 {\setminus} \{\ve\}$, $\Bw$ is a locally compact Hausdorff space. Also, $\Bb{\ve} = X \times \T$ is a locally compact Hausdorff space. Thus, since $\EE$ has the disjoint union topology, it follows that $\EE$ is a locally compact Hausdorff groupoid.
\end{proof}

\subsection{Defining the quotient groupoid \texorpdfstring{$\GG$}{G} and proving that \texorpdfstring{$\EE \to \GG$}{E} is a twist \texorpdfstring{}{over G}}
\label{subsection: quotient groupoid and twist}

In this section we show that $\EE$ is a twist. We begin by defining the quotient groupoid $\GG$. Let
\[
\GG \coloneqq X \rtimes_\alpha \F_2 = \{ (\alpha_w(x),w,x) : x \in X, \, w \in \F_2 \} \subseteq X \times \F_2 \times X.
\]
Then $\GG$ is a groupoid under the multiplication and inversion operations
\[
(x,w,u)(u,w',y) \coloneqq (x,ww',y) \quad \text{ and } \quad (x,w,y)^{-1} \coloneqq (y,w^{-1},x).
\]
The collection
\[
\big\{ \{ (\alpha_w(x),w,x) : x \in U \} : w \in \F_2, \, U \text{ is an open subset of } X \big\}
\]
is a basis for a locally compact Hausdorff \'etale groupoid topology on $\GG$. The unit space of $\GG$ is $\Go = \{ (x,\ve,x) : x \in X \}$, which we identify with $X$. The range and source maps are given by $r(x,w,y) = x$ and $s(x,w,y) = y$. Since $\alpha_a = \sigma$ is minimal, $\GG$ is minimal. The $1$-cocycle $c_\GG\colon \GG \to \F_2$ defined by $c_\GG(\alpha_w(x),w,x) \coloneqq w$ is continuous and $\F_2$ is discrete, so $\Gw \coloneqq c_\GG^{-1}(w)$ is clopen in $\GG$ for each $w \in \F_2$. We now show that $\EE$ is a twist over $\GG$.

\begin{prop} \label{proposition: twist}
Define $\imath\colon X \times \T \to \EE$ by $\imath(x,z) \coloneqq (\ve,(x,z))$, and define $\pi\colon \EE \to \GG$ by $\pi(w,c) \coloneqq \big(\alpha_w(\pw(c)), w,\pw(c)\big)$. Then $X \times \T \xrightarrow{\imath} \EE \xrightarrow{\pi} \GG$ is a twist over $\GG$.
\end{prop}

\begin{proof}
By \cref{proposition: EE groupoid}, $\EE$ is a locally compact Hausdorff groupoid. We have $\Eo = \imath(X \times \{1\})$, and $\imath$ is a continuous groupoid homeomorphism onto the open set $\imath(X \times \T) = \Eb{\ve} = \pi^{-1}(\Go)$ that restricts to a homeomorphism of unit spaces. To see that $\pi$ is a homomorphism, fix $\big((w,c),(w',c')\big) \in \Ec$. Then $(c,c') \in \Cb{w,w'}$, so $\pw(c) =\alpha_{w'}\big( \pb{w'}(c')\big)$. Write $cc' \coloneqq \psi_{w,w'}([c,c'])$. By \cref{equation: right-fibre-preserving multiplication}, $\pb{ww'}(cc') = \pb{w'}(c')$, and so
\begin{align*}
\pi(w,c) \, \pi(w',c') &= \big(\alpha_w(\pw(c)),w,\pw(c)\big) \big(\alpha_{w'}(\pb{w'}(c')),w',\pb{w'}(c')\big) \\
&= \big(\alpha_w(\pw(c)),ww',\pb{w'}(c')\big) = \big(\alpha_{ww'}(\pb{w'}(c')),ww',\pb{w'}(c')\big) \\
&= \big(\alpha_{ww'}(\pb{ww'}(cc')),ww',\pb{ww'}(cc')\big) = \pi(ww', cc') = \pi\big((w,c)(w',c')\big),
\end{align*}
and hence $\pi$ is a groupoid homomorphism. Since $\pw$ is a continuous surjection, $\pi$ is a continuous surjection. For all $x \in X$, $\pi(\ve,(x,1)) = (x,\ve,x)$, so $\pi$ restricts to a homeomorphism of unit spaces. To see that $\pi$ is an open map, first note that for each $w \in \F_2$, it follows by local triviality of the principal $\T$-bundle $\pw\colon \Bw \to X$ that $\pw$ is an open map (using an argument similar to the proof of \cite[Lemma~2.7(a)]{Armstrong2022}). So for each fixed $w \in \F_2$ and any open set $V \subseteq \Bw$, we have $\pi(\{w\} \times V) = \{ (\alpha_w(x),w,x) : x \in \pw(V) \}$, which is open in $\GG$ because $\pw(V)$ is open in $X$. Since $\EE$ has the disjoint union topology, it follows that $\pi$ is an open map. To see that the extension is central, fix $x, y \in X$, $z \in \T$, and $(w,c) \in \EE$ such that $\big(\imath(x,z),(w,c)\big), \big((w,c),\imath(y,z)\big) \in \Ec$. Then $x = \alpha_w\big(\pw(c)\big) = \alpha_w(y)$, and by \cref{lemma: reducing the products}\cref{item: BFP identity}, $\psi_{\ve,w}\big([(x,z),c]\big) = z \cdot c = \psi_{w,\ve}\big([c,(y,z)]\big)$. Thus,
\[
\imath(x,z) \, (w,c) = (\ve,(x,z))(w,c) = (w, z \cdot c) = (w,c)(\ve,(y,z)) = (w,c) \, \imath(y,z),
\]
so $\imath(X \times \T)$ is central in $\EE$.
\end{proof}

\begin{remark} \label{remark: pi restricted to Eb is nontrivial}
For $(x,w,y) \in \GG$, we have $\alpha_w(x) = y$. Hence $(x,w,y) \mapsto (w,y)$ is a homeomorphism $\GG \to \F_2 \times X$ that carries each $\Gw$ onto $\{w\} \times X$. Thus $s_b\colon \GGb \ni (x,b,x) \mapsto x \in X$ and $j_b\colon \Bb{b} \ni c \mapsto (b,c) \in \Eb{b}$ are homeomorphisms that satisfy $s_b \circ \pi\restr{\EEb} \circ j_b = \pb{b}$, where $\pb{b}\colon \Bb{b} \to X$ is our original nontrivial principal $\T$-bundle $p\colon B \to X$ from \cref{subsection: nontrivial T-bundle}.
\end{remark}

\subsection{Studying the isotropy and proving the desired properties}

In this section we prove that the twist of \cref{proposition: twist} is not induced by a $2$-cocycle, even when restricted to the interior of the isotropy. Recall that we write $\IG$ for the topological interior of the isotropy $\Iso(\GG)$ of $\GG$. Define the ``$a$-counting map'' $\ell_a\colon \F_2 \to \Z$ to be the homomorphism defined on generators $a, b \in \F_2$ by
\[
\ell_a(a) = 1 \quad \text{ and } \quad \ell_a(b) = 0.
\]

\begin{lemma} \label{lemma: isotropy}
For the groupoid $\GG = X \rtimes_\alpha \F_2$ defined in \cref{subsection: quotient groupoid and twist}, we have
\[
\IG = \Iso(\GG) = \bigsqcup_{w \in \ker(\ell_a)} \! \Gw \,\simeq\, X \times \ker(\ell_a) \quad \text{ and } \quad \IE = \pi^{-1}(\IG) = \Iso(\EE) = \bigsqcup_{w \in \ker(\ell_a)} \! \Ew.
\]
\end{lemma}

\begin{proof}
By definition,
\[
\Iso(\GG) = \{ (\alpha_w(x),w,x) : x \in X, \, w \in \F_2, \, \alpha_w(x) = x \}.
\]
Since the $\T$-bundle $\pb{b}\colon \Bb{b} \to X$ is nontrivial, $X$ is not discrete, and in particular, every dense subset of $X$ is infinite. Thus, if $\sigma^{\ell_a(w)}(x) = \alpha_w(x) = x$, then minimality of $\sigma$ implies that $\ell_a(w) = 0$ (for if $\ell_a(w) > 0$ then the orbit of $x$ under $\sigma$ is finite). Thus
\[
\Iso(\GG) = \{(x,w,x) \in \GG : x \in X, \, w \in \F_2, \, \ell_a(w) = 0\} = \bigsqcup_{w \in \ker(\ell_a)} \! \Gw
\]
is homeomorphic to $X \times \ker(\ell_a)$. Since $\Gw = c_\GG^{-1}(w)$ is open in $\GG$ for each $w \in \F_2$, $\Iso(\GG)$ is a union of open sets and is therefore open, and so $\IG = \Iso(\GG)$.

By the proof of \cite[Corollary~2.11(b)]{Armstrong2022}, we have $\IE = \pi^{-1}(\IG) = \pi^{-1}(\Iso(\GG)) = \Iso(\EE)$. Since $\pi^{-1}(\Gw) = \Ew$ for each $w \in \F_2$, it follows that
\[
\IE \,=\, \pi^{-1}(\IG) \,=\, \bigsqcup_{w \in \ker(\ell_a)} \! \pi^{-1}(\Gw) \,=\, \bigsqcup_{w \in \ker(\ell_a)} \! \Ew. \qedhere
\]
\end{proof}

\begin{lemma}
The restricted twist $\pi\restr{\IE}\colon \IE \to \IG$ is not induced by a $2$-cocycle.
\end{lemma}

\begin{proof}
Suppose for contradiction that $\pi\restr{\IE}\colon \IE \to \IG$ is induced by a $2$-cocycle. Then $\pi\restr{\IE}$ admits a continuous global section $\Sigma\colon \IG \to \IE$. By \cref{lemma: isotropy}, we have $\GGb \subseteq \IG$ and $ \EEb = \pi^{-1}(\GGb) \subseteq \IE$. So $\Sigma\restr{\GGb}$ is a continuous global section of $\pi\restr{\EEb} \colon \EEb \to \GGb$. By \cref{remark: pi restricted to Eb is nontrivial}, this gives a continuous global section of the principal $\T$-bundle $p\colon B \to X$, contradicting that the bundle is nontrivial.
\end{proof}

\vspace{2ex}
\end{document}